\documentclass[12pt]{article}
\usepackage{amsmath,amssymb,amsthm, amsfonts,tcolorbox}
\usepackage{hyperref}
\hypersetup{
     colorlinks=true,
     linkcolor=black,
     filecolor=black,
     citecolor = black,      
     urlcolor=cyan,
     }

\usepackage[margin=1in]{geometry} 
\usepackage[nameinlink]{cleveref}
\usepackage{graphicx,color}

\theoremstyle{definition}
\usepackage{url}
\usepackage{dsfont} 
\usepackage{color}
\definecolor{red}{rgb}{1,0,0}

\definecolor{blue}{rgb}{0,0,1}

\definecolor{green}{rgb}{0,.6,0}

\usepackage{float}
\usepackage{lineno}
\usepackage{tabularx}
\usepackage{tikz,pgfplots}
\usetikzlibrary{decorations.pathreplacing,calligraphy}
\usetikzlibrary{shapes.geometric}
\usetikzlibrary{shapes.multipart}
\usetikzlibrary{arrows}

\usetikzlibrary{calc}
\usetikzlibrary{arrows.meta}
\usetikzlibrary{arrows}
\usetikzlibrary{patterns}
\usetikzlibrary{decorations.markings}

\newtheorem{thm}{Theorem}[section]
\newtheorem{cor}[thm]{Corollary}
\newtheorem{lem}[thm]{Lemma}
\newtheorem{prop}[thm]{Proposition}
\newtheorem{conj}[thm]{Conjecture}
\newtheorem{obs}[thm]{Observation}

\newtheorem{exmp}[thm]{Example}
\newtheorem{con}[thm]{Construction}
\newtheorem{claim}[thm]{Claim}

\theoremstyle{remark}

\theoremstyle{definition}

\theoremstyle{definition}

\newcommand{\floor}[1]{\lfloor #1 \rfloor}

\newcommand{\N}{\mathbb{N}}

\newcommand{\rb}{\operatorname{rb}}

\newcommand{\bit}{\begin{itemize}}
\newcommand{\eit}{\end{itemize}}
\newcommand{\ben}{\begin{enumerate}}
\newcommand{\een}{\end{enumerate}}
\newcommand{\beq}{\begin{equation}}
\newcommand{\eeq}{\end{equation}}
\newcommand{\bea}{\begin{eqnarray}} 
\newcommand{\eea}{\end{eqnarray}}
\newcommand{\bpf}{\begin{proof}}
\newcommand{\epf}{\end{proof}\ms}
\newcommand{\bmt}{\begin{bmatrix}}
\newcommand{\emt}{\end{bmatrix}}
\newcommand{\ms}{\medskip}

\title{Note on vertex disjoint rainbow triangles in edge-colored graphs}

\author{ J\"urgen Kritschgau\thanks{ Fariborz Maseeh Department of Mathematics and Statistics, Portland State University, Portland, OR, USA (jkritsch@pdx.edu). Research is supported by NSF RTG grant DMS-2136228.} \and tahda queer\thanks{Dept.~of Mathematics, Hunter College, City University of New York, NY, USA (taqueer@proton.me)} \and Cyrus Young\thanks{Dept.~of Mathematics, University of California, Irvine, CA, USA (cyrusjy@uci.edu)} \and Wohua Zhou\thanks {Dept.~of Mathematics,  California State University, East Bay, CA, USA (wzhou31@horizon.csueastbay.edu)}}

\date{July 10, 2023}

\pgfplotsset{compat=1.18} 

\begin{document}
\maketitle

\begin{abstract}

  Given an edge-colored graph $G$, we denote the number of colors as $c(G)$, and the number of edges as $e(G)$.
  An edge-colored graph is rainbow if no two edges share the same color.
  A proper $mK_3$ is a vertex disjoint union of $m$ rainbow triangles.
  Rainbow problems have been studied extensively in the context of anti-Ramsey theory, and more recently, in the context of Tur\'{a}n problems.
  B. Li. et al. 
  \cite{li2014rainbow} 
  found that a graph must contain a rainbow triangle if  $e(G)+c(G) \geq \binom{n}{2}+ n$.
  L. Li. and X. Li. 
  \cite{li2022edge}
  conjectured a lower bound on $e(G)+c(G)$ such that $G$ must contain a proper $mK_3$. 
  In this paper, we provide a construction that disproves the conjecture. 
  We also introduce a result that guarantees the existence of $m$ vertex disjoint rainbow $K_k$ subgraphs in general host graphs, and a sharp result on the existence of proper $mK_3$ in complete graphs.

\end{abstract}


\section{Introduction}

All graphs in this paper are finite simple graphs.
Given a graph $G$, we denote the vertex set and edge set of $G$ by $V(G)$ and $E(G)$, respectively.
An \emph{edge-colored graph} is a tuple $(G, C)$, where $G$ is a graph and $C$ is a mapping from $E(G)$ to $\N$.
We use $C(e)$ to denote the color of an edge $e \in E(G)$, and $C(G)$ to denote the set of colors of all edges in $G$.
We let $e(G) = |E(G)|$ and $c(G) = |C(G)|$.
A clique on $n$ vertices is denoted $K_n$.
If we have two graphs $G$ and $H$, the \emph{join} of $G$ and $H$, denoted $G \vee H$, is the graph obtained by taking the disjoint union of $G$ and $H$ and adding an edge $uv$ for every $u \in V(G)$ and every $v \in V(H)$.

A  graph $G$ is called \emph{rainbow} if no two edges share the same color.
A collection of subgraphs $\{G_i\}_{i \in I}$ of a graph is called \emph{vertex-disjoint} if for any $i,j \in I, i\neq j$ implies $V(G_i)\cap V(G_j)=\emptyset$. 
We say a graph $G$ is a \emph{proper $mK_3$} if $G$ is a vertex-disjoint collection of rainbow $K_3$. Note that by definition, a rainbow triangle is a rainbow $K_3$ and a proper $K_3$.
For any terminology used but not defined here, we refer to \cite{IntroGraphTheory}.

Rainbow matching problems can be generalized to rainbow $K_3$ problems. 
By Tur\'an's theorem, any graph $G$ on $n$ vertices with $c(G)> \floor{n^2/4}$ must have a rainbow $K_3$. 
B. Li. et al. were able to improve this result to include $e(G)$ with the following result:

\begin{thm}(Theorem 1 in \cite{li2014rainbow}).\label{thm: m=1}
    Let $G$ be an edge-colored graph on $n$ vertices. If $e(G) + c(G) \geq \binom{n+1}{2}$, then $G$ contains a rainbow triangle.
\end{thm}

This theorem was generalized to all complete graphs. 
Let $t_{n,k}$ denote the number of edges in the complete $k$-partite graph on $n$ vertices with balanced part sizes. 

\begin{thm}(Theorem 6 in \cite{XU2016193})\label{thm:cliques}
    Let $G$ be an edge-colored graph on $n$ vertices and $n\geq k\geq 4$. If $e(G)+c(G)\geq \binom{n}{2} +t_{n,k-2}+2,$ then $G$ contains a rainbow $K_k$.
\end{thm}

Let $\rb(G,H)$ denote the minimum number of colors required such that every edge coloring of $G$ with at least $\rb(G,H)$ colors contains a rainbow $H$.

\begin{cor}(Corollary 1 in \cite{XU2016193})
\label{cor:cliques}
Let $G$ be an edge-colored graph on $n$ vertices and $n\geq k \geq 3$. 
If $e(G)+c(G)\geq \binom{n}{2} + \rb(K_n,K_k)$, then $G$ contains a rainbow $K_k$.
\end{cor}

A related line of research is finding minimum color degree conditions that guarantee rainbow subgraphs. 
Color degree sum conditions for rainbow triangles have also been studied in \cite{li2016color}. 
J. Hu, H. Li, and D. Yang \cite{hu2020vertex} found minimum color degree conditions that guarantee vertex disjoint rainbow triangles and L. Li and X. Li \cite{li2022vertex} found degree sum conditions for disjoint rainbow cycles.  

Our work in this paper concerns finding a lower bound on $e(G) + c(G)$ for any graph $G$, such that any graph with a greater number of combined edges and colors must have a proper $mK_3$. 
Notably, this question is fairly well understood for edge disjoint rainbow triangles.
L. Li and X. Li \cite{li2022edge} showed that if $e(G)+c(G)\geq \binom{n+1}{2} + 3k-3$, then $G$ contains $k$ edge disjoint rainbow triangles. 
Furthermore, it was conjecture that if $G$ is a graph on $n \geq 5m$ vertices with $e(G) + c(G) \geq \binom{n+1}{2} +6m -6$, then $G$ contains a proper $mK_3$ (see \cite{li2022edge}). 
We disprove this conjecture with a new construction which inspires the following conjecture.

\begin{conj}\label{conj:main1}
    Let $m \geq 1$. For an edge-colored graph $G$ on $n \geq 5m +2$ vertices, if
    $$e(G) + c(G) > \binom{n}{2} +mn -\binom{m+1}{2},$$
    then $G$ admits a proper $mK_3$.
\end{conj}

In Section \ref{sec: 2}, we will describe the construction which informs our conjecture (see Construction~\ref{con: m>1}). 
We will also introduce a new result that provides a lower bound on $e(G) + c(G)$ which guarantees that $G$ contains a proper $mK_k$.
Erd\H{o}, Simonovits, and S\'{o}s showed that if $k\geq 4$, then  any edge-colored complete graph on $n-k(m-1)$ vertices with  $t_{n-k(m-1),k-2}+2$ color contains a rainbow $K_k$ (see \cite{erdos10anti}).
In other words, $\rb(K_{n-k(m-1)},K_{k})=t_{n-k(m-1),k-2}+2$ for $k\geq 4$ and $\rb(K_{n-3(m-1)}, K_3)= n$.

\begin{prop}\label{thm:main1}
    Let $k \geq 3$ and $m \geq 1$. For an edge-colored graph on $n \geq km$ vertices, if 
    $$e(G)+c(G) \geq \binom{n -k(m-1)}{2} + \rb(K_{n-k(m-1)}, K_{k}) + k(m-1)(2n-1) -k^2(m-1)^2,$$
    then $G$ contains $m$ vertex disjoint rainbow $K_k$ subgraphs.
\end{prop}

This proposition is proven with a standard induction argument, and is not close to the bound in Conjecture~\ref{conj:main1}.
Finally, for complete host graphs, we have the following bound on $c(G)$.

\begin{thm}\label{thm:main2}
    Let $m \geq 1$. For a complete edge-colored graph $G$ on $n \geq 9m+8$ vertices, if 
    $$c(G) > mn -\binom{m+1}{2},$$
    then $G$ contains a proper $mK_3$.
\end{thm}

Note that by Construction~\ref{con: m>1}, the bound on the number of colors is best possible. 
However, we conjecture that the bounds on $n$ can be improved (see Conjecture~\ref{conj:main1}).

Interestingly, the minimum number of colors in a complete graph that guarantees a rainbow $mK_3$ is roughly $n^2/4$, as opposed to the roughly $mn$ colors needed to find a proper $mK_3$ (see \cite{wu2023anti} for details). 


\section{Lower Bound Constructions}\label{sec: 2}

In this section, we describe our construction, which disproves the conjecture proposed by L. Li and X. Li in \cite{li2022edge}.

First, we use $T_n$ to denote the edge-colored complete graph with vertex set $\{v_1, v_2, \dots , v_n \}$ and coloring $C:V\to \{1,\dots, n-1\}$ given by $C(v_iv_j)=i$ for $1 \leq i < j \leq n$. 
This construction is used in \cite{li2014rainbow} by B. Li et al. to prove that Theorem \ref{thm: m=1} is sharp. 
Notice that $T_n$ does not contain any rainbow triangle, as any triangle $v_i v_j v_k$ with $1 \leq i < j < k \leq n$ has edges $C(v_iv_j)=C(v_i v_k)=i$. 
Furthermore, $e(T_n)+c(T_n)=\binom{n}{2}+n-1.$

L. Li and X. Li described a proper $mK_3$-free graph in \cite{li2022edge}.
Let $G$ be a complete graph on $n$ vertices, which contains $G_0=T_{n-5m+5}$ with $V(T_{n-5m+5})= \{ v_1, v_2, \dots, v_{n-5m+5}\}$. 
For $1 \leq i \leq m-1$, let $K_5$ be a rainbow complete graph vertex disjoint from $G_{i-1}$, and let $ G_i = G_{i-1} \vee K_5$. 
Each of the edges in $K_5$ receives a color that is not in $C(G_{i-1})$, and the edges in $E_{G_i}(G_{i-1}, K_5)$ receive a single new color.
Notice that $c(G)=n-5m+4+10(m-1)+(m-1)=n+6m-7$.

We now describe our construction, which contains more colors.

\begin{con} \label{con: m>1}
    Let $G$ be an edge-colored complete graph on $n$ vertices that contains $T_{n-m+1}$.
    Every edge outside of $T_{n-m+1}$ receives a unique color, giving $\binom{n}{2}-\binom{n-m+1}{2}$ new colors.
    The total number of colors is $n-m+\binom{n}{2}-\binom{n-m+1}{2} = mn -\binom{m+1}{2}$.
    In this construction,
      \begin{equation*}
         e(G)+c(G) =\binom{n}{2} + mn -\binom{m+1}{2}.
      \end{equation*}
 \end{con} 
 
    Construction \ref{con: m>1} does not contain a proper $mK_3$, because $T_{n-m+1}$ is free of rainbow triangles, which means any rainbow triangle in $G$ has to use at least one vertex from the $m-1$ vertices outside of $T_{n-m+1}$. 
    By the pigeon hole principle, for any $m$ rainbow triangles, there are two triangles that share a vertex. 
    Therefore, $m$ copies of vertex disjoint rainbow triangles cannot exist in this construction.

When $m=1$, the constructions mentioned above are the same as $T_n$.
When $m > 1$, given $n \geq 5m$, the sum of the number of edges and colors in Construction \ref{con: m>1} is always larger than the construction raised in the end of \cite{li2022edge}, because
$\binom{n}{2}+ mn -\binom{m+1}{2}> \binom{n}{2}+n+6m-7$ when $m > 1$ and $n \geq 7+ \frac{m}{2}$.

There is a third construction, also a complete graph with no proper $mK_3$, which has more colors than Construction \ref{con: m>1} when $n < 5m-2$.
Let $G$ be a join of $T_{n-3m+1}$ and a rainbow $K_{3m-1}$, and the colors in $T_{n-3m+1}$ are not used in the rainbow $K_{3m-1}$. 
We then assign all of the edges between $T_{n-3m+1}$ and $K_{3m-1}$ to single new color. 
Notice that $c(G)=n-3m+\binom{3m-1}{2}+1= n+2 +\frac{9m^2-15m}{2} \geq \frac{m(2n-m-1)}{2}$ when $n \leq 5m-2$.

 Construction~\ref{con: m>1} is the basis for Conjecture~\ref{conj:main1}. 
 This conjecture is sharp for $n \geq 5m-2$ if it holds.


\section{Proof of Proposition~\ref{thm:main1}}

In this section, we prove Proposition~\ref{thm:main1}. We use an induction proof that accounts for the maximum number of edges and colors lost when removing a rainbow triangle.

\begin{proof}[Proof of Proposition~\ref{thm:main1}]

    Let
    $$f(k,m,n) = \binom{n -k(m-1)}{2} + \rb(K_{n-k(m-1)}, K_{k}) + k(m-1)(2n-1) -k^2(m-1)^2.$$
    A quick calculations gives $f(k,m,n) - f(k,m-1,n-k) = 2k(n-k) +k(k-1)$

    \begin{claim}
        For any edge-colored graph $G$, if $G$ contains a rainbow $K_k$, denoted $H$, then $e(G) + c(G) - e(G-H) -c(G-H) \leq 2k(n-k) +k(k-1)$. 
    \end{claim}
    
    \begin{proof}
        There are $n-k$ vertices in $G -H$, so there can be at most $n-k$ edges between any vertex in $A$ and $G -H$, for a maximum of $k(n-k)$ edges.
        Including the edges in $H$, we have $e(G)-e(G -H)\leq k(n-k) + \binom{k}{2}$.  
        Since the number of colors in $G$ is always at most the number of edges, we have $e(G) + c(G) - e(G-H) -c(G-H) \leq 2[e(G) -e(G-H)] \leq 2[k(n-k) +\binom{k}{2}] = 2k(n-k) +k(k-1)$.
    \end{proof}

    We now proceed by induction on $m$.

    \emph{Base Case:} let $m=1$, and let $G$ be a graph on $n \geq k$ vertices. Then
    $$e(G) + c(G) \geq f(k,1,n) = \binom{n}{2} + \rb(K_{n-k(m-1)}, K_{k}).$$
    By Theorem~\ref{cor:cliques}, $G$ contains a rainbow $K_k$.

    \emph{Induction Step:} Assume for all $m' < m$ and graphs $G'$ with $n \geq km'$ vertices, if $e(G) + c(G) \geq f(k,m',n)$, then $G'$ contains $m'$ vertex-disjoint rainbow $K_k$. 
    Let $G$ be a graph on $n\geq mk$ vertices with $e(G) + c(G) \geq f(k,m,n)$. 
    By Theorem~\ref{cor:cliques}, $G$ contains a rainbow $K_k$, denoted $A$. 
    Then 
    $$e(G-A) + c(G-A) \geq f(k,m,n) -2k(n-k) -k(k-1) = f(k,m-1,n-k).$$
    Notice that $G-A$ is a graph on $n\geq k(m-1)$ vertices. By induction, $G-A$ contains $m-1$ vertex disjoint rainbow $K_k$, implying $G$ contains $m$ vertex disjoint rainbow $K_k$.
\end{proof}


\section{Proof of Theorem~\ref{thm:main2}}

Let $A \subseteq V(G)$.
A color $R$ is \emph{saturated by $A$} if for every edge $e$ colored $R$, $e \cap A \neq \emptyset$.
We use $d^s_G(A)$ to denote the number of colors in $G$ saturated by $A$.
Notice that $d^s_G(A) = c(G) - c(G-A)$.

We say a color $R$ is \emph{ideally saturated by $A$} if $R$ is saturated by $A$ but not saturated by any proper subset of $A$.
\begin{obs}\label{obs:saturation}
A color  $R$ is saturated by $A$ iff $R$ is ideally saturated by some subset of $A$.
\end{obs}
We use $\varphi_G(A)$ to denote the number of colors in $G$ that are ideally saturated by $A$. 
Note that $\varphi_G(\emptyset) = d^s_G(\emptyset) = 0$, and $\varphi_G(\{v\}) = d^s_G(\{v\})$ for any vertex $v \in V(G)$.

\begin{lem}\label{lem1*}
    Let $G$ be a graph, and $A \subseteq V(G)$.
    \begin{enumerate}
        \item[(1)] $d^s_G(A) \leq \sum_{B\subseteq A}{\varphi_G(B)}$.
        \item[(2)] If $A$ ideally saturates $R$, then every vertex in $A$ sees an edge-colored $R$.
    \end{enumerate} 
\end{lem}

\begin{proof}
    \emph{(1)} Let $X$ be the set of colors in $G$ saturated by $A$, and $Y =\{(R,B): B \subseteq A$ and $R$ is ideally saturated by $B\}$. 
    Then $|X| = d^s_G(A)$ and $|Y| = \sum_{B \subseteq A}{\varphi_G(B)}$. 
    We define an injective function $f: X \rightarrow Y$, implying $|X| \leq |Y|$.
    For each color $R$ that is saturated by $A$, we can find a subset $B_R\subseteq A$ so that $R$ is ideally saturated by $B_R$ by Observation~\ref{obs:saturation}. Therefore, the function  $f$  given by $R\mapsto (R,B_R)$ exists and is injective. 
    
    \emph{(2)} If there is a vertex $v$ that does not see an edge-colored $R$, then  $A - \{v\}$ must saturate $R$; which is to say that $A$ does not ideally saturate $R$. 
    This is the contrapositive of the claim.
\end{proof}

Furthermore, suppose we have $N \geq 1$ distinct vertices $v_1, ..., v_N \in V(G)$. 
We define $\varphi_G(v_1, ..., v_N) := \sum_{i=1}^N{\varphi_G(\{v_j: 1 \leq j \leq i\})}$.

\begin{lem}\label{lem2*}
    If $\varphi_G(v_1, ..., v_N) \geq k + N -1$, then there are  at least $k$ edges $e_1, ..., e_k \in E(G)$ that satisfy the following. 
    \begin{enumerate}
        \item[(1)] For every $i \leq k$, $e_i \cap \{v_1, ..., v_N\} = \{v_1\}$.
        \item[(2)] $C(e_i) \neq C(e_j)$ when $i \neq j$.
        \item[(3)] Every $C(e_i)$ is saturated by some set $\{v_j: 1 \leq j \leq M\}$ for $M \leq N$.
    \end{enumerate}
\end{lem}

\begin{proof}
    Let $A_i = \{v_1, ..., v_i\}$ and $B_i = \{R: A_i$ ideally saturates $R\}$ for $M \leq N$. 
    Then $\varphi_G(A_i) = |B_i|$. 
    By definition, we know $B_i \cap B_j = \emptyset$ when $i \neq j$.
    
    We know $v_1 \in A_i$ for every $i$.
    Then by Lemma \ref{lem1*} \emph{(2)}, for every $i$ and for every $R \in B_i$, there is an edge incident to $v_1$ colored $R$. 
    There are $\sum_{i=1}^N{|B_i|} = \varphi_G(v_1, ..., v_N)$ edges $e_1, ..., e_{\varphi_G(v_1, ..., v_N)}$, each incident to $v_1$, satisfying \emph{(2)} and \emph{(3)}.

    All edges incident to $v_1$ can be sorted into two disjoint sets: 
    \begin{align*}
        W &= \{ \{v_1,v\} : v \in A_N \setminus\{v_1\}\},\\
        Z &= \{ \{v_1,v\} : v \in V(G) \setminus A_N\}.
    \end{align*} 
    Note that the edges in $Z$ are the edges satisfying \emph{(1)}.
    Since $|W| = N-1$, it follows that there must be at least $k$ edges $e_i$ in $Z$.
\end{proof}

\begin{proof}[Proof of Theorem \ref{thm:main2}]

The case for $m=1$ is proved by Theorem \ref{thm: m=1}. Let $m\geq 2$.
Assume that  $G$ is a complete graph on $n \geq 9m+8$ vertices 
and that $G$ does not contain a proper $mK_3$.
Consider the following algorithm:
    
\begin{enumerate}
    \item Set $G_0:=G$ and $i:=0$.
        
    \item If there exists $N$ vertices $v_1, ..., v_N$ for $N \leq 3$ where $\varphi_{G_i}(v_1, ..., v_N) > 3(m-i) +N$, let $G_{i+1} = G_i -\{v_1\}$, increment $i=i+1$, and return to step 2.

    \item Else, if there exists a set of vertices $\{u, v, w\}$ in $G_i$ with $uvw$ being a rainbow triangle and $d^s_{G_i}(\{u, v, w\}) \leq n-i-1$, let $G_{i+1} = G_i -\{u, v, w\}$, increment $i=i+1$, and return to step 2.
      
    \item Return $k=i$.
\end{enumerate}

\begin{claim}\label{claim:add.vertex}
    For $0 \leq i \leq k$, $G_i$ contains a proper $(k-i)K_3$.
\end{claim}

\begin{proof}
We will prove the claim by reverse induction on $i$. 
The claim is trivially true if $i=k$. 
Assume the claim is true for $1 < i \leq k$. 
Let $H$ be the established proper $(k-i)K_3$ in $G_i$. 
There are two cases; either $G_i = G_{i-1} -\{v\}$ or $G_i = G_{i-1} -\{u,v,w\}$.

\textbf{Case 1:} 
Assume $G_i = G_{i-1} -\{v\}$ and let $v_1=v$. 
This implies that $\varphi_{G_{i-1}}(v_1,\dots, v_N)\geq 3(m-i)+N+1$ with $1\leq N\leq 3$.
By Lemma~\ref{lem2*} (taking $k=3(m-i)+1$), we know there are  edges $e_i$ with $1\leq i \leq 3(m-i)+2$ incident upon $v_1$ with distinct colors which are saturated by $\{v_1,\dots,v_N\}$. 
Let $X$ be the set of neighbors of $v_1$ in $G_{I-1}$ that use some edge $e_i$. 
Note that by Lemma~\ref{lem2*}, $v_2,\dots,v_N\not\in X$. 
Since $H$ has $3(k-i)$ vertices, there exists $u, w \in X\setminus V(H)$. Since the colors of $vu$ and $vw$ are saturated by $\{v_1,\dots, v_N\}$ and $u,w\not\in\{v_1,\dots, v_N\}$, it follows that $uvw$ is a rainbow triangle that is vertex-disjoint from $H$. 
Then $H \cup uvw$ is a proper $(k-i+1)K_3$.

\textbf{Case 2:} 
Assume $G_i = G_{i-1} -\{u,v,w\}$. 
By construction, $uvw$ is a rainbow triangle in $G_{i-1}$ that is vertex-disjoint from $H$. 
Then $H \cup uvw$ is a proper $(k-i+1)K_3$.
\end{proof}

From this claim, we can infer that $k<m$.

\begin{claim}\label{claim:triangle.sat}
    There is no rainbow triangle in $G_k$, and $c(G_k) \leq n-k-1$
\end{claim}

\begin{proof}
Assume that $G_k$ contains a rainbow triangle $uvw$.
According to the algorithm, for any $N \leq 3$ vertices $v_1, ..., v_N$, $\varphi_{G_k}(v_1, ..., v_N) \leq 3(m-k) +N$. 
Therefore, we have $\varphi_{G_k}(u,v),\varphi(v,w)_{G_k}\leq 3(m-k)+2$ and $\varphi_{G_k}(w,u,v)\leq 3(m-k)+3$.
Furthermore, $d^s_{G_k}(\{u,v,w\}) > n-k-1$.

By Lemma~\ref{lem1*}, the bounds in the previous paragraph, and the assumption that $n\geq 9m+8$, we have 
\begin{align*}
    n-k-1 
    <&~ d^s_{G_k}(\{u,v,w\})\\ 
    \leq&~ \sum_{B \subseteq \{u,v,w\}}\varphi_{G_k}(B)\\
    =&~\textcolor{red}{\varphi_{G_k}(\{u\})}  + \textcolor{blue}{\varphi_{G_k}(\{v\})} +\varphi_{G_k}(\{w\}) \\
    &+ \textcolor{red}{\varphi_{G_k}(\{u,v\})}+\textcolor{blue}{\varphi_{G_k}(\{v,w\})}+\varphi_{G_k}(\{w,u\}) +\varphi_{G_k}(\{ w,u,v\})\\
    =&~ \textcolor{red}{\varphi_{G_k}(u,v)} + \textcolor{blue}{\varphi_{G_k}(v,w)}+ \varphi_{G_k}(w,u,v)\\
    \leq&~ 9(m-k)+7\\
    \leq&~n -k -1
\end{align*}
which is a contradiction. 
By Theorem \ref{thm: m=1}, $k \geq 1$ and $c(G_k) \leq n-k-1$.
\end{proof}

We now assume that $G$ does not contain a proper $mK_3$, and will show that $c(G)\leq mn-\binom{m+1}{2}$. 
Let $W(G_i)= c(G_{i-1}) - c(G_{i})$.

\begin{claim}
    For all $1 \leq i \leq k$, we have $W(G_i) \leq n-i.$
\end{claim}

\begin{proof}
Once again, there are two cases.

\textbf{Case 1:} 
If $G_i = G_{i-1} -v$, then $W(G_i) = d^s_{G_{i-1}}(\{v\}) \leq |V(G_{i-1})| -1 \leq n-i$.

\textbf{Case 2:} 
If $G_i = G_{i-1} -\{u,v,w\}$, then $W(G_i) = d^s_{G_{i-1}}(\{u,v,w\}) \leq n-i$.
\end{proof}

Now we have
\begin{align*}
    c(G)&= \sum_{i=1}^{k}W(G_i) +c(G_k)\leq \sum_{i=1}^{k+1} [n-i] \leq \sum_{i=1}^{m} [n-i]= mn - \binom{m+1}{2}.
\end{align*}
This completes the proof.
\end{proof}

There are some difficulties in generalizing the proof of Theorem~\ref{thm:main2} to non-complete host graphs. 
The trouble in generalizing to non-complete host graphs (and thus addressing Conjecture~\ref{conj:main1} directly) arises in Case 1 of the proof of Claim~\ref{claim:add.vertex}. 
In this case, we assume that $G_i=G_{i-1}-\{v\}$ and that $G_i$ contains a proper $(k-i)K_3$ subgraph. 
The fact that we can extend this $(k-i)K_3$ subgraph to an $(k-i+1)K_3$ in $G_{i-1}$ by adding a vertex $v$ and two of its neighbors along edges colored with a saturated color. 
However, for this argument to work, we need that the neighbors of $v$ along edges of saturated colors are themselves connected by edges, which is true in complete graphs. 
It was not clear to us how to modify step 2. in the greedy algorithm used in the proof of Theorem~\ref{thm:main2} to deal with a non-complete host graph. 

Step 2. of the algorithm would also need to be modified in order to generalize our argument to find proper $mK_r$ for $r>3$. 
First, step 2. would need to be more sophisticated in order to control the number of colors that proper (or rainbow) cliques saturate as was done in the proof of Claim~\ref{claim:triangle.sat}. 
Second, it is not clear that deleting a single vertex that saturate many colors as in step 2. will guarantee that proper or rainbow $(m-i)K_r$ can be extended when the deleted vertex is added back to the graph.


\section*{Acknowledgments}
This research was conducted in Summer 2023 as a part of the Summer Undergraduate Applied Mathematics Institute at Carnegie Mellon University. We thank Dr. David Offner and Dr. Michael Young for their support, as well as all other people in the program for being amazing. We also thank CMU for funding this project and hosting SUAMI.

J.K. is supported by the National Science Foundation through NSF RTG grant DMS-2136228.


\bibliographystyle{plainurl}
\bibliography{paper.bib}

\end{document}